\documentclass[preprint,12pt]{elsarticle}
\usepackage{url}

\usepackage{graphics}

\usepackage{amssymb}
\usepackage{amsmath}
\usepackage{amsthm} 

\newtheorem{thm}{Theorem}
\newtheorem{lemma}{Lemma}

\begin{document}

\begin{frontmatter}


\title{On an optimal potential of Schr\"odinger operator with prescribed $m$ eigenvalues}





\author[im,im_Br]{Y.Sh.~Ilyasov\corref{cor1}}\ead{ilyasov02@gmail.com}
\author[im]{N.F.~Valeev}\ead{valeevnf@yandex.ru}

\cortext[cor1]{Corresponding author}

\address[im]{Institute of Mathematics of UFRC RAS, 45008, 112, Chernushevskogo, Ufa, Russia;}
\address[im_Br]{Instituto de Matem\'atica e Estat\'istica.
	Universidade Federal de Goi\'as,	74001-970, Goiania, Brazil}

\begin{abstract}

The purpose of this paper is twofold: firstly, we present a new type of relationship between inverse problems and nonlinear differential equations. Secondly, we introduce a new type of inverse spectral problem, posed as follows: for a priori given potential $V_0$ find  the closest  function $\hat{V}$ such that $m$ eigenvalues of one-dimensional space  Schrodinger operator with potential $\hat{V}$  would coincide with the given values $ E_1 $,$ \ldots $, $ E_m \in \mathbb {R} $. In our main result, we prove the existence of a solution to this problem, and more importantly, we show that such solution can be directly found by solving a system of nonlinear differential equations.

\end{abstract}


\end{frontmatter}






\section{Introduction}
The inverse problems are some of the most paramount problems in science and are present in everything around us. Our sensory contact with the world around us depends on an intuitive solution of an inverse problem. The shape, size, and color of external objects are inferred from their scattering and absorption of light as detected by our eyes. At the same time, the modern theory of vision assumes that perception is not only passive reception of signals, but it is also shaped by the learning,  including processing information associated with already existing concepts and knowledge of a person. 
Evidently, the case, whereby some preliminary information about the unknown parameters of the model exists, is typical for applied inverse problems. This raises the following \textit{inverse optimal problem}: 
\par
\textit{Determine the model parameters $\hat{F}$ which are the closest to the a priori given dates $F_0$ and which produce the observed measurements $S$.}

Below we will show on an example of Schr\"odinger equation that this problem is solvable. However, before this let us mention that this problem is related to the following direct problem:
\par
\textit{How to transform a given system $F_0$ under minimal changes to a new one $\hat{F}$ with pre-set properties $S$?}
\par\noindent
This problem, to a certain extent, can be defined as a problem of minimum fine-tuning, which is understandably important in itself and arises in many applications 
(see e.g. \cite{Zakhariev}).





\section{Main result}
  We illustrate the inverse optimal problem by considering  the time - independent Schr\"odinger wave equation one space dimension 
\begin{equation} \label{eq:S}
H_V\psi :=-\psi''+V(x)\psi=E \psi\,\,\,\,
\end{equation}
in finite interval $(0,L)$ subject to the zero boundary condition
\begin{eqnarray}
	\psi(0)=\psi(L)=0.
	\end{eqnarray}
 We assume that  $V \in L^2(\Omega)$. 
 Under these conditions   $H_V$ defines a self-adjoint operator on the Hilbert space  $L^2(0,L)$ (see, e.g., \cite{zetl}), so that its spectrum consists of an infinite sequence of eigenvalues $\sigma_p(H_V):=\{E_i(V) \}_{i=1}^{\infty}$ which can be ordered as follows  $E_1(V)<E_2(V)< \ldots$.
Furthermore, to each eigenvalue $E_k(V)$ corresponds an unique (up to a normalization constant) eigenfunction $\phi_k(V)$  which has exactly $k-1$ zeros in $(0,L)$. In what follows, we shall always suppose that  $\|\phi_k(V)\|_{L^2}=1$, $k=1,2,...$.
Here and what follows, we denote by $\left\langle \cdot, \cdot \right\rangle $ and $\|\cdot\|$   the scalar product and the norm in  $L^2:=L^2(0,L)$; $W^{1,2}(0,L)$, $W^{2,2}(0,L)$ are usual Sobolev spaces and 
 $W^{1,2}_0:=W^{1,2}_0(0,L)$ is the closure of $C^\infty_0(0,L)$ in the norm
$
\|\psi\|_{1}=\left(\int^L_0 |\nabla \psi |^2 dx\right )^{1/2}
$.

Let $m\geq 1$. We  study the following $m$-parameter inverse optimal spectral problem:

\medskip

\par\noindent
(P):\,\,\textit{For a given $E_1,\ldots, E_m \in \mathbb{R}$ and $V_0 \in L^2:=L^2(0,L)$, find   a potential  $\hat{V} \in L^2$ such that}
\begin{align}
	&\bullet ~E_1=E_1(\hat{V}),\ldots, E_m=E_m(\hat{V}),\nonumber \\
	&\bullet ~\|V_0-\hat{V} \|^2_{L^2}=\min_{ V \in L^2}\left\{\|V_0-V\|^2_{L^2}: E_k=E_k(V), k=1,\ldots,m \right\}.\label{minProb}
\end{align}

It turns out that  this problem is related to the solving the following system of nonlinear  equations:
\begin{equation}
\tag{$\mathcal{E}$}
\label{eq:Nonl}
\left\{
\begin{aligned}
  -u_i''&+V_0 u_i=E_i u_i+\sum_{j=1}^m \sigma_j u_j^{2}u_i,~~&&i=1,2,...,m, \\
  ~~u_i&(0)=u_i(L)=0,~~ &&i=1,2,...,m. \\
\end{aligned} \right.
\end{equation}
where $\sigma_i\in\{0,+1, -1\}$, $i=1,\ldots,m$. 

%
\begin{thm}\label{thm1}
Let $E_1,\ldots, E_m \in \mathbb{R}$ and $V_0 \in L^2$ are given, $m\geq 1$. 
Then there exists  a solution  $\hat{V}$ of inverse optimal spectral problem $(P)$.

Furthermore,  $\hat{V}$ is expressed in terms of solution of system  $\eqref{eq:Nonl}$ that is \eqref{eq:Nonl} possesses a non-zero weak  solution $(\hat{u}_1,..., \hat{u}_m) \in  (C^2(0,L)\cap C^1[0,L])^m$ for some constants $\sigma_1,...,\sigma_m \in \{0,+1, -1\}$   so that
 \begin{equation}\label{OPEN}
	\hat{V}(x)=V_0(x)-\sum_{j=1}^m \sigma_j \hat{u}_j^{2}(x).
\end{equation}
Moreover, in the case  $ E_{i'}\neq E_{i'}(V_0)$ for some $i' \in \{1,\ldots m\}$,  $\sigma_1,...,\sigma_m $  are not  equal simultaneously to zero, i.e.,  $\sum_{j=1}^m|\sigma_j|\neq 0$.
\end{thm}
Notice that if $E_i=E_i(V_0)$,  $\forall i=1,\ldots, m$, then the solution of $(P)$ is \textit{trivial}, i.e., $\hat{V}=V_0$ and $\sum_{j=1}^m|\sigma_j|=0$ in \eqref{OPEN}.

\section{Proof of the main result}

In the next lemma, we derive a key formula in our approach
\begin{lemma}\label{lem1}
For $k\geq 1$, the map 	$E_k(\cdot): L^2 \to \mathbb{R}$  is continuously
differentiable with the Fr\'echet-derivative
\begin{equation}\label{eq:Val}
	D_VE_k(V)(h)=\frac{1}{\|\phi_k(V)\|^2}\left\langle \phi_k^2(V), h\right\rangle, ~~\forall \,V ,h \in L^2.
\end{equation}
\end{lemma}
\begin{proof} Since $E_k(V)$ is isolated,  Corollary 4.2 from \cite{ZettlM} implies that
$E_k(V)$ is Fr\'echet differentiable and \eqref{eq:Val} holds. By the analyticity property (see  \cite{Poschel}, page 10), the map $\phi_k(\cdot): L^2 \to W^{2,2}(0,L)$ is analytic.
Due to the Sobolev theorem,
the embedding $W^{2,2}(0,L) \subset L^4(0,L)$ is continuous. Hence the map $\phi_k(\cdot): L^2 \to L^4(0,L)$ is continuous and therefore the norm of the derivative $D_VE_1(V)$ continuously depends on $q \in L^2$. This implies  that $E_k(V)$ is continuously differentiable in $L^2$.
\end{proof}

Denote 
$$
\hat{Q}=\min\left\{\|V_0-V\|^2:~~ E_k=E_k(V),~ k=1,\ldots,m,~~V \in L^2\right\}.
$$
Let $(V_j)\subset L^2$ be a minimizing sequence of \eqref{minProb}.
Due to coerciveness of $||V_0-V||^2$ on  $L^2$,  the sequence $(V_j)$ is bounded in $L^2$, and therefore the Banach-Alaoglu theorem yields that there exists a subsequence which we again denote by $(V_j)$ such that $||V_0-V_j|| \to 
||V_0-\hat{V}||$ and $
V_j \rightharpoondown \hat{V}
$
weakly in  $L^2$ for some $\hat{V} \in L^2$.

%

Let $i=1,\ldots,m$. Consider the sequence of the eigenfunctions $(\phi_i(V_j))$. The boundedness of  $(V_j)$ in $L^2$ and the assumption $E_i(V_j)\equiv E_i$ for $j=1,\ldots,$ implies (see \cite{ilValUMJ}) that the sequences $(\phi_i(V_j))$ is bounded in $W^{2,2}(0,L)$. From this and by the Sobolev embedding theorem there exists a subsequence which we again denote by $(\phi_i(V_j))$, $i=1,\ldots,m$  such that
\begin{equation} \label{eq:S5}
\phi_i(V_j) \to \phi_i^*~~~\mbox{as}~~~j\to +\infty, ~~\mbox{strongly in}~~W^{1,2}(0,L) \cap C^1[0,L],
\end{equation}
 for some $\phi_i^*\in W^{1,2}_0(0,L) \cap C^1[0,L]$. Notice that, since $\|\phi_i(V_j))\|=1$,   $\forall j=1,2,...$, it follows that $\phi_i^*\neq 0$.
Moreover,
\begin{align*}\label{eq:S6}
\phi_i(V_j) = E_i(V_j) \int_0^L G_0&(x,\xi)(\phi_i(V_j)(\xi)-\phi^*_i(\xi))d\xi \\
&-\int_0^L G_0(x,\xi) V_j(\xi) (\phi_i(V_j)(\xi)-\phi^*_i(\xi))d\xi+ \nonumber\\
+&E_i(V_j) \int_0^L G_0(x,\xi)\phi^*_i(\xi)d\xi -\int_0^L G_0(x,\xi) V_j(\xi) \phi^*_i(\xi)d\xi.
\end{align*}
where $G_0(x,\xi)$  is the integral kernel of operator $H_V^{-1}|_{V=0}$.
In view of that $G_0 \in C[0,L] \times C[0,L]$, the strong convergences \eqref{eq:S5} and the weak convergence $V_j \rightharpoondown \hat{V}$ in $L^2$ imply
\begin{equation} \label{eq:S8}
 \phi^*_i(x)=E_i \int_0^L G_0(x,\xi)\phi^*_i(\xi)d\xi -\int_0^L G_0(x,\xi) \hat{V}(\xi) \phi^*_i(\xi)d\xi,
\end{equation}
or in the equivalent form
\begin{equation} \label{eq:S7}
H_{\hat{V}}\phi^*_i(x)=E_i \phi^*_i(x).
\end{equation}
Thus $(E_i, \phi^*_i)$  coincides with some eigenpairs of the operator $H_{\hat{V}}$, that is there exist  $r_1,\ldots, r_m \in \mathbb{N}$ such that
\begin{equation}
	E_i=E_{r_i}(\hat{V}),~~~~\phi^*_i=\phi_{r_i}(\hat{V}), ~~~i=1,\ldots,m,
\end{equation}
Let us show that
$ r_i=i$ for every $i=1,\ldots,m$.
By the Sturm comparison theorem (see e.g.,\cite{zetl}) for each $j=1,2, ... $, the eigenfunction $\phi_i(V_j)(x)$, $i=1,\ldots,m$ has precisely $i-1$ roots in $(0,L)$.
This and the  convergences \eqref{eq:S5} in $C^1[0,1]$ yield that the limit function $\phi^*_i$, $i=1,\ldots, m$ may has at most $i-1$ roots in $(0,L)$.
Hence  $r_i \leq i-1$, $i=1,\ldots, m$. 
Since $\left\langle \phi_i(V_j),\phi_k(V_j)\right\rangle=0$ for all $j=1,2,...$ and $i\neq k$, by passing to the limit we deduce that $\left\langle \phi_i^*,\phi_k^*\right\rangle=0$ for any $i,k \in \{1,\ldots,m\}$ such that $i\neq k$. Thus $\phi_i^*\neq \phi_k^*$ and since  $r_i \leq i-1$,  $i=1,\ldots, m$, we obtain  $ r_i=i$ for every $i=1,\ldots,m$. Thus 
$E_i=E_i(\hat{V})$,~ $\forall i=1,\ldots,m$ that is $\hat{V}$ is an admissible point for minimization problem \eqref{minProb}. From this and since by  the weak  convergence  $V_j \rightharpoondown \hat{V}$ in $L^2$ one has 
$
||V_0-\hat{V}|| \leq \liminf_{j\to \infty}||V_0-V_j||=\hat{Q}
$,we conclude that $||V_0-\hat{V}||= \hat{Q}$. Thus  $\hat{V}$ is a minimizer of \eqref{minProb}.

Let us show \eqref{OPEN}. In view of that $E_k(\cdot): L^2 \to \mathbb{R}$, $k=1,\ldots, m$  are continuously
differentiable functions, we may apply the Lagrange multiplier rule which yields the equality
\begin{equation}
	\mu_0 D_V(||V_0-\hat{V}||^2)(h)+\sum_{j=1}^{m}\mu_j D_VE_j(\hat{V})(h) =0,~~ \forall h \in L^2,
\end{equation}
for some $\mu_0,\ldots, \mu_m \in \mathbb{R} $ such that $|\mu_0|+\ldots+|\mu_m|\neq 0$. Thus by \eqref{eq:Val} we deduce
\begin{equation}
	\int_\Omega \left(-2\mu_0  (V_0-\hat{V})+\sum_{j=1}^{m}\mu_j\phi_j^2(\hat{V})\right) h\, dx  =0, \,\, \forall h \in L^2,
\end{equation}
where $\|\phi_2(\hat{V})\|=1$. Hence,
$$
2\mu_0  (V_0-\hat{V})=\sum_{j=1}^{m}\mu_j\phi_j^2(\hat{V})~~\mbox{a.e. in}~~(0,L).
$$
Observe that
 $\mu_0 \neq 0$. Indeed, if $\mu_0=0$, then $\sum_{j=1}^{m}\mu_j\phi_j^2(\hat{V})(x)=0$ a.e. in $(0,L)$. However this is impossible since $(\phi_j^2(V))_{j=1}^m$, for any $m\geq1$, forms a system of independent  functions in $[0,L]$ (see below Appendix).   Suppose that $\sum_{j=1}^{m}|\mu_j|=0$. Then $V_0=\hat{V}$  a.e. in $(0,L)$ and consequently $E_j=E_j(V_0)$ for every $j=1,\ldots,m$  which contradicts our assumption. Thus $\sum_{j=1}^{m}|\mu_j|\neq 0$ and 
\begin{equation*}
	\hat{V}(x)=V_0(x)+\sum_{j=1}^m \mu_j \phi_j^{2}(\hat{V})~~\mbox{a.e. in}~~\Omega.
\end{equation*}
 Substituting this into the equalities
$$
 -\phi_i''(\hat{V})+\hat{V} \phi_i(\hat{V})=E_i \phi_i(\hat{V}),~~i=1,\ldots,m,
$$
we obtain  
\begin{equation*}\label{eq:phi}
- \phi_i''(\hat{V})+V_0 \phi_i(\hat{V})=E_i \phi_i(\hat{V})-(\sum_{j=1}^m \mu_j \phi_j^{2}(\hat{V}))\phi_i(\hat{V}),~~j=1,\ldots,m.
\end{equation*}
Thus, the functions $\hat{u}_i=|\mu_i|^{1/2}\phi_i(\hat{V})$, $i=1,\ldots, m$ satisfy to  system \eqref{eq:Nonl} and we have proved \eqref{OPEN} .

\section{Conclusion and remarks}

The first part of Theorem \ref{thm1} only demonstrates the mere existence of a solution. Indeed,   one should not expect in the general to find the explicit form of the functions $E_k(V)$, $k=1,\ldots $. However, equation \eqref{eq:Nonl}  can be solved, at least numerically, and thus the optimal potential in $(P)$ can de obtained by  \eqref{OPEN} (up to know the values of the constants ($\sigma_i$)).

The one-parameter inverse optimal spectral problem, that is when only one eigenvalue $E_m$ is predetermined in (P), has been  studied in our recent papers \cite{ilValMatZam, ilVal, ilValUMJ} where the existence of an inverse optimal potential $\hat{V}$ has been proven. Furthermore, in this case, the stronger result holds, namely: uniqueness of the inverse optimal potential $\hat{V}$ and uniqueness of the solution of the corresponding nonlinear equation \eqref{eq:Nonl}   are satisfied. Moreover,   the constant $\sigma_1$ in the corresponding equation \eqref{eq:Nonl}  is exactly determined. 

The one-parameter inverse optimal spectral problem $(P)$ for the $N$ -dimensional space Schr\"odinger equation can be solved when $m=1$, that is in the case of pre determinedness of the principal eigenvalue $E_1$ (see \cite{ilVal}). In this case, the key formula \eqref{eq:Val} is still valid, i.e.  $E_1(\cdot): L^2(\Omega) \to \mathbb{R}$  is continuously
differentiable and its  Fr\'echet-derivative is given by \eqref{eq:Val}.

The dual to one-parameter inverse optimal spectral problem (P), i.e. finding the maximal value of the first eigenvalue $E_1$ with the prescribed value of the norm in $L^p$  of the potential $V$, has been considered in \cite{Zhang}, where an equation similar to \eqref{eq:Nonl} has been obtained as well. 


Apparently, there should be no difficulties in generalizing Theorem \ref{thm1} to the cases of the Schr\"odinger equation considered on the whole interval $(-\infty,+\infty)$ or with other types of boundary conditions.

We conjecture that the solution $\hat{V}$ to the problem (P) is unique and, as in the case of one-parameter inverse optimal spectral problem (see \cite{ilValMatZam, ilVal, ilValUMJ}), the values of constants $ (\sigma_i)_{i=1}^m $ in equation \eqref{eq:Nonl} are uniquely determined. In this regards, 
notice that the relationship between problems (P) and \eqref{eq:Nonl} gives rise to another rather interesting problem on the uniqueness of solutions of nonlinear differential equations.

\section{Appendix}
 Let $\phi_i$,  $i=1,...,m$ be eigenfunctions of the operator $H_V$. We show that the system of functions
$\left\lbrace \phi^2_i(x) \right\rbrace_{i=1}^{m}$ 
   is linearly independent in $[0,L]$. 
The proof is by induction on $m$. Assume the statement is true for $m-1$; we will prove it  for $m$. 
  
 To obtain a contradiction suppose that   
\begin{equation}\label{eq1:u_k}
\sum_{j=1}^m \alpha_k \phi_i^{2}(x)= 0,~~\forall x \in [0,L],
\end{equation}
for some $\alpha_1,\alpha_2,...,\alpha_m$ such that $\sum_{i=1}^m |\alpha_i| \neq 0$. 
Differentiating \eqref{eq1:u_k}  twice  we get
$$
 \sum_{i=1}^m \alpha_i \phi_i(x)\phi_i''(x)+ \sum_{j=1}^m \alpha_i (\phi_i'(x))^2 = 0
$$
Hence, using the equalities $H_V\phi_i=E_i\phi_i$,  $i=1,...,m$ and \eqref{eq1:u_k} we derive
 $\sum_{j=1}^m \alpha_i ((\phi_i'(x))^2-E_k \phi_i^2(x))= 0$. 
Now once more differentiating  we obtain 
$\sum_{j=1}^m \alpha_i (\phi_i''(x)\phi_i'(x)-E_k \phi_i(x)\phi_i'(x) )= 0$. 
 Hence by  $H_V\phi_i=E_i\phi_i$,  $i=1,...,m$ and \eqref{eq1:u_k} we derive
$\sum_{j=1}^m \alpha_i E_k \phi_i(x)\phi_i'(x) = 0$.
In view of the assumptions $\phi_i (0) = 0$,  $i=1,...,m$, this implies
$\sum_{j=1}^m \alpha_i E_k \phi_i^{2}(x)= 0 $. This and \eqref{eq1:u_k} yield that 
$ \sum_{j=1}^{m-1} \gamma_k \phi_i^{2}(x)= 0$ for some constants  $\gamma_1,\gamma_2,...,\gamma_{m-1}$. But this contradicts to the induction assumption.


\end{document}